\documentclass[a4paper,UKenglish]{lipics-v2016}

\usepackage{microtype}

\usepackage{todonotes}

\usepackage{algorithm}
\usepackage{algpseudocode}

\usepackage{thmtools,thm-restate}

\title{Greedy Morse Matching and Discrete Smoothness}

\author[1]{Joao Lagoas}
\author[2]{Thomas Lewiner}
\author[3]{Tiago Novello}
\author[4]{Joao Paixao}

\affil[1]{Colégio Pedro II, Rio de Janeiro, Brazil\\
  \texttt{jaolagoas@gmail.com}}
\affil[2]{Gamma, The Boston Consulting Group, Paris, France\\
  \texttt{thomas@lewiner.org}}
\affil[3]{Pontifícia Universidade Católica do Rio de Janeiro, Brazil\\
  \texttt{tiago.novello@mat.puc-rio.br}}
\affil[4]{Universidade Federal do Rio de Janeiro, Brazil\\
  \texttt{jpaixao@dcc.ufrj.br}}
\authorrunning{J.Lagoas et al.} 

\Copyright{J.Lagoas, T. Lewiner, T. Novello and J. Paixao}


\begin{document}

\maketitle

\begin{abstract}
Discrete Morse theory emerged as an essential tool for computational geometry and topology. Its core structures are discrete gradient fields, defined as acyclic matchings on a complex $C$, from which topological and geometrical information of $C$ can be efficiently computed, in particular its homotopy, homology, or Morse-Smale decomposition.

On the geometrical side, given a function $f$ sampled on $C$, it is possible to derive a discrete gradient field that mimics the dynamics of $f$. Many such constructions are based on some variant of a greedy pairing of adjacent cells, given an appropriate weighting. However, proving that the dynamics of $f$ is correctly captured by this process is usually intricate. This work introduces a notion of discrete smoothness of the pair $(f,C)$, as a minimal sampling condition to ensure that the discrete gradient is geometrically faithful to $f$.
More precisely, a discrete gradient construction from a function $f$ on a simplicial complex $C$ of arbitrary dimension is studied, leading to theoretical guarantees prior to the discrete smoothness assumption. Those results are then extended and completed for the smooth case.

On the topological side, given an appropriate function $f$, greedy matchings can also be used to construct optimal discrete gradient field to provide topological information of a complex $C$. As an application, a purely combinatorial proof that all CAT(0) cube complexes are collapsible is given.
\end{abstract}

\section{Introduction} \label{sec:introduction}
Discrete Morse theory, as introduced by Forman~\cite{Forman95,Forman98}, provides foundations for tools in both computational topology and topological data analysis. This ubiquity comes from its combinatorial nature that translates to algorithms in a straightforward manner. In particular its core structure, called discrete gradient vector field, is defined through combinatorial matching, which allows to leverage a large literature of efficient algorithms.

This combinatorial nature lead to several applications of topological data analysis,
where a common goal is to analyse the dynamics of a scalar function $f$ defined on the vertices of a cell complex $C$. For such setup, greedy matching have been widely used to construct a discrete gradient field on $C$ out of $f$~\cite{babson2005discrete,King_Hnudson_Mramor,lewiner2004applications,Gyulassy,robins2011theory}\footnote{Non-greedy matching approaches have also been devised, for example fusing progressive simplification~\cite{reininghaus2011fast} or locally enforced properties~\cite{robins2011theory,adiprasito2011metric}.}. The general principle of those greedy approaches relies on matching adjacent cells, weighing valid pairs aligned with the direction of steepest descent of $f$. This leads to efficient, scalable and generalizable algorithms.

However, while Forman's theory ensures the topological coherence of the resulting discrete gradient field, the later is not always guaranteed to be faithful to $f$: its critical elements are not necessarily facets incident to piecewise-linear critical vertices, as defined by Banchoff~\cite{banchoff1967critical}, even in the simplicial case. For specific cases, when the weights guiding the greedy matching are given by a lexicographic ordering~\cite{babson2005discrete} or a linear interpolation of $f$~\cite{lewiner2013critical},proofs of the faithfulness of the construction have been given. However those proofs are very intricate and require initial barycentric subdivisions~\cite{babson2005discrete,lewiner2013critical} to ensure enough regularity of the sampling of the function $f$.

This work pins down the notion of regularity needed by greedy constructions, which we call \emph{discrete smoothness}. This leads to simpler and broader proofs of the faithfulness of a greedy discrete gradient field construction. Moreover, we show that barycentric subdivisions ensures discrete smoothness, the proposed result generalizes previous ones~\cite{babson2005discrete,lewiner2013critical}.

More precisely, given a simplicial complex $C$ and a scalar function $f$ sampled at its vertices, we introduce the notion of discrete smoothness of the pair $(f,C)$, as a minimal sampling condition to ensure that a greedy discrete gradient field can be faithful to $f$. The condition is necessary and sufficient for the critical elements to be well positioned by the greedy construction. 

We structure the following study around a greedy algorithm where weights of matching adjacent cells are taken from lexicographic ordering of their vertices according to $f$, following Babson and Hersh's work~\cite{babson2005discrete}. When applied on a smooth pair $(f,C)$ of any finite dimension, this construction leads to isolated critical elements, similarly to smooth Morse functions.

The simplicity of greedy matchings have also been useful to pure computational topology~\cite{forman_user, adiprasito2011metric}. As an application of our theoretical approach to greedy matchings, we provide a purely combinatorial proof of the result of Adiprasito and Benedetti~\cite{adiprasito2011metric} that all CAT(0) cubical complexes are collapsible.







\section{Greedy Matching} \label{sec:greedy_algorithm}

Most concepts of graph theory and their notations have been derived from \cite{Bondy_and_Murty}.

A graph $ G $ is an ordered pair $ (V(G), E(G)) $ consisting of a set of vertices $ V(G) $ and a set of edges $ E(G) $ disjoint of $ V(G) $. Each edge of $ G $ is associated with an unordered pair of vertices of $ G $. For $ x, y \in {V (G)} $, an edge $ e = \{x, y \} \in {E (G)} $ is denoted by $ xy $ or $ yx $ without distinction. We say that $ x $ and $ y $ are \textit{extremes} of the edge $e$ and that $ e $ is \textit{incident} to vertices $ x $ and $y$.

A path $P$ in a graph $G$ is a sequence of vertices denoted by $ v_1, v_2, \dots, v_n $ such that from each of its vertices there is an edge to the next vertex in the sequence. In addition, the edges of a graph may have an associated weight. In a weighted graph, every edge is associated with a real and finite number which we will call \textit{weight}. We will define $ \overline{e}$ or $ \overline{ \{x, y\} }$ as the weight associated with the edge $ e = \{x, y \} \in {E}$.

A graph $H$ is a subgraph of $G$ denoted by $ H \subseteq{G}$ if $V(H) \subseteq{V (G)}$ and  $E(H)\subseteq{E(G)}$. We say that a subgraph $ H $ is an induced subgraph of $ G $ and denoted by $ G [S] $, where $ S $ is a subset of vertices of $ G $, if for any pair of vertices $ x $ and $ y $ of $ H $, $ xy $ is an edge of $ H $ if and only if $ xy $ is an edge of $ G $. This means that $ H $ is an induced subgraph of $ G $ if it has all the edges that are in $ G $ on the same set of vertices.

A matching $ M $ in a graph $ G = (V, E) $ is a subset of edges $ M \subseteq {E (G)} $ such that every vertex is contained in at most one edge of $ M $ . Edges in $ M $ are called \textit{matched edges} and the extreme vertices to the matching edges are called \textit{saturated vertices}. Similarly, edges that are not in $ M $ are said to be \textit{unmatched} and the vertices where all the edges are not matched are called \textit{unsaturated vertices}.

An alternating path $ P $ in a matching $ M $, defined in a graph $ G $, is a path whose edges alternate between those that are in $ M $ and those that are not in $ M $. $ P $ can therefore be defined as $ P = \{e_1, e_2, e_3, \dots, e_n\} $ where $ n $ is the number of edges in the path, that is, its length. Therefore, if $ e_i \in {M} $ then $ e_{i + 1} \notin M$ for all $i = 1,2,3, \dots, n-1 $.

A digraph $ D $ is an ordered pair $ (N (D), A (D)) $ consisting of a non-empty set of nodes $ N (D) $ and a set of arcs $ A (D) $ disjoint of $ N (D) $. At each arc of $ D $ an ordered pair of nodes is associated. A digraph is nothing more than a directed graph; the pairs of nodes that make up an arc are now ordered. For $ u, v \in N $, an arc $ a = \{u, v \} \in {A} $ is denoted by $ uv $ and implies that $ a $ is directed from $ u $ to $ v $. As in graphs, a digraph can be weighted, with all its arcs associated with a finite and real number.

The main algorithm from matching theory that we will use is simply a greedy matching in graphs. We rewrite its steps in pseudocode form as shown in the next algorithm.

\label{alg:greedy_matching}
\begin{algorithm}[H] 

  \textbf{Input:} \text{Weighted graph $G=(E, V)$ } 
  \textbf{Output:} \text{Set of matched edges $M$}

  \caption{Greedy matching}
  \begin{algorithmic}
    \State $M\leftarrow \emptyset$ \Comment{An empty set of edges}
    \While{$E\neq\emptyset$}\Comment{There are still edges in the graph}
      \State $E_{min}\leftarrow \text{argmin}_{e\in E(G)} \overline{e}$
      \State $M\leftarrow M\cup{E_{min}}$
      \State $G\leftarrow G$ after removal of the extreme vertices to all the edges in $E_{min}$
    \EndWhile
    \State \textbf{return} $M$\Comment{The matching}
  \end{algorithmic}
\end{algorithm}

At each iteration, the algorithm selects the edges with the lowest weights in the graph by creating the set of minimum edges $E_{min}$, adds them to the matching $M$ and finally removes the vertices that are extreme from these edges and, therefore, their incident edges. Thus, at each iteration an induced subgraph of $G$ is created. These steps of selecting edges, matching, and removing vertices repeat until all edges in $E$ have been removed.



At this moment it is worth noting that $M$ is only a matching if the intersecting edges (neighboring edges) of the graph are tie free. This means that in $G$ there is no edge $e_1$ adjacent to edge $e_2$ such that $\overline{e}_1 =\overline{e}_2$. It is fundamental to understand this sufficient condition because it will be used throughout the paper.

We now propose and define one notation that will be useful in order to simplify the lemmas and proofs related to the theoretical guarantees of this algorithm.

First, it is necessary to remember that we denote the $-:e\rightarrow\mathbb{R}$ as function evaluated on an edge $e$ as the weight of that edge $\overline{e}$. We are ratifying this because we will define another function, also named as $-$, but defined in the vertices of the graph induced by a matching $M$.

Thus, for a matching $M$ in $G$, the function $-:V(G)\rightarrow\mathbb{R}\cup\infty$ is such that 
$$
\overline{v}= \left
\{
\begin{matrix}
\overline{e},  \mbox{ if } v \mbox{ is saturated by } e \mbox{ in } M 
\\
\infty , \mbox{ if } v \mbox{ is not saturated in } M
\end{matrix}\right.
$$

We will call the function $-$ therefore by \textit {saturation function}. The saturation function is well defined because $M$ is a matching, so each vertex can only be saturated by at most one edge. In addition, the matching type will always be implied in context.

Now we will define $V_a = \{v \in {V (G)} | \overline{v} \geq{a} \}$ as the set of vertices of $G$ whose saturation function is greater than or equal to the real number $a$. Then consider $G_a = G[V_a]$ as the induced subgraph of $G$ for the set of vertices $V_a$. This induced subgraph contains only the vertices whose saturation function is greater than or equal to the real number $a$. We denote therefore the set of vertices and the set of edges of $G_a$ as $V(G_a) = V_a$ and $E(G_a) = E_a$, respectively. 

By definition, $G_a$ is a subgraph of $G$ where the minimum weight of its edges is $a$. This notation is interesting because it identifies the subgraph of $G$ at the moment that the edges of a certain weight are selected in the matching step of the algorithm. In other words, we can recognize the iteration in which the algorithm is executing when we refer to $G_a$.

With this in mind, it is also easy to note that if a weight edge $e$ exists in the graph $G_{\overline{e}}$ then surely that edge is matched in $M$. The next lemma formalizes this idea.

\begin{lemma} \label{lem:is_matched}
Let $M$ be a greedy matching in a graph $G$. If $e \in G_{\overline{e}}(E)$, then $e\in M$.
\end{lemma}

The greedy matching has some properties in relation to its set of edges $M$. These properties are strongly related to a particular type of alternating path and will be useful in assisting various proofs on the discrete gradient further along. For that, we need to define some particular cases of alternating paths. Here is the definition of what we call the \textit {maximal alternating path} (see examples in Figure~\ref{fig:maximal_alternating_path}).

\theoremstyle{definition}
\begin{definition} [Maximal Alternating Path] \label{def:maximal_alternating_path}
We denote as \textit{maximal alternating path}, any alternating path $P$ such that if a vertex is saturated, then it is saturated by an edge in $M$.
\end{definition}

\begin{figure}[ht] 
    \centering
    \includegraphics[scale=0.35]{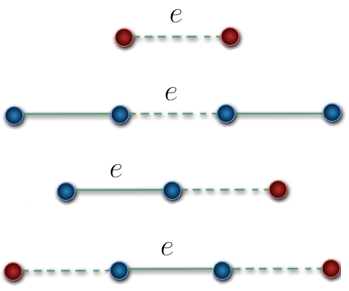}
    \caption{Examples of maximal alternating paths. Red and blue vertices are unsaturated and saturated, respectively. Solid edges are matched and dotted edges are not.}
    \label{fig:maximal_alternating_path}
\end{figure}

An important result of the greedy algorithm is the notion of occurrence of edges of minimum weight on maximal alternating paths. By means of the Theorem~\ref{thm:if_map_then_matched}, we can say that for any maximal alternating path, the edges of minimum weight will be matched.

\begin{theorem} \label{thm:if_map_then_matched}
If $P$ is a maximal alternating path, then $E_{min}(P) \subseteq {M} $.
\end{theorem}

It is important to note that a maximal alternating path $P$ such that $|P| = 1$ and its vertices are unsaturated, deserves some attention. Note that this type of construction never happens in a greedy matching. Although not so intuitive, note that this is verifiable by the last theorem.


\section{Preliminaries for the Discrete Gradient Field}
\label{sec:discrete_gradient_field}

\theoremstyle{definition}
\begin{definition} [Simplicial Complex] \label{def:simplicial_complex}

A finite simplicial complex is a set of vertices $ V $ together with a set $ \Delta $ of subsets of $ V $, such that $ \Delta $ must satisfy some properties:

\begin{enumerate}
  \item $V\subset{\Delta}$
  \item if $\tau \in \Delta$ and $\sigma \subset \tau$, then $\sigma \in \Delta$
\end{enumerate}
\end{definition}

We will refer to the simplicial complex as $ \Delta $. The elements of $ \Delta $ are called simplexes. A simplex $ \tau \in \Delta $ is said of \textit{dimension} $ p $, denoted by $ dim(\tau) = p $, if $ \tau $ contains $ p + 1 $ vertices.

A simplex $ \sigma $ is also said to be a \textit{facet} of another $\tau$ simplex, denoted by $ \sigma \prec \tau $, if $ \sigma \subset \tau $ and $dim(\sigma)=dim(\tau)-1$. We finally say that the dimension of $ \Delta $ is the largest dimension of these simplexes.

Another very common way to represent a simplicial complex is in terms of a digraph. This structure is called the \textit{Hasse diagram} and is defined below \cite{Chari}.

\theoremstyle{definition}
\begin{definition} [Hasse Diagram] \label{def:hasse_diagram}
Given a simplicial complex $\Delta$, its Hasse diagram $H$ is the directed graph $D$ where the set of nodes is the set of simplexes of $\Delta$ and the set of arcs in $H$ is composed of all pairs such that $\{\sigma,\tau\}\in H$ if and only if $\sigma\prec\tau$.
\end{definition}

The Hasse diagram, therefore, acts as a kind of dictionary between the concepts of graphs and simplicial complexes. It is through its structure that we can understand how the discrete gradient is constructed.

\theoremstyle{definition}
\begin{definition} [Discrete Vector Field] \label{def:discrete_vector_field}
A discrete vector field $\mathcal{V}$ in $\Delta$ is a collection of simplexes pairs $\{\sigma,\tau\}$ in $\Delta$ with $\sigma\prec\tau$, such that each simplex is at most a pair of $\mathcal{V}$.
\end{definition}

We write $\sigma\rightarrow\tau$ if $\{\sigma,\tau\}\in\mathcal{V}$, that is, if the arc $\{\sigma,\tau\}$ is matched,  and $\sigma\nrightarrow\tau$ otherwise. 

Unsaturated nodes and alternate paths in the Hasse diagram are also defined in a particular way.

\theoremstyle{definition}
\begin{definition} [Critical Simplex]\label{def:critical_simplex}
A simplex $\sigma$ is said \textit{critical} for a discrete vector field if it does not belong to any pair of $\mathcal{V}$.
\end{definition}

\theoremstyle{definition}
\begin{definition} [$\mathcal{V}-path$] \label{def:v_path}
Given a discrete vector field $\mathcal{V}$ in a simplicial complex $\Delta$, a $\mathcal{V}-path$ is a sequence of simplexes $\blacktriangleleft\sigma_0\tau_0\sigma_1\tau_1\dots\sigma_{n-1}\tau_{n-1}\sigma_n \blacktriangleright$ such that for $0 \leq i \leq n-1$ we have that $\sigma_i\rightarrow\tau_i$ and $\tau_i\succ\sigma_{i+1}\neq\sigma_i$.
\end{definition}

A $\mathcal{V}-path$ is then an alternating path in the Hasse diagram with the constraint that all dimensions of the simplexes alternate between $p$ or $p+1$, with $p=dim(\sigma_n)$.

If a path is such that $n>0$ and $\sigma_n=\sigma_0$, we call it \textit{non-trivial closed}. We say that $\mathcal{V}$ is a discrete gradient field if it does not contain any $\mathcal{V}-path$ non-trivial closed.

\theoremstyle{definition}
\begin{definition} [Discrete Gradient Field] \label{def:discrete_gradient_field}
A discrete vector field $\mathcal{V}$ is a discrete gradient field if and only it does not contain non-trivial closed $\mathcal{V}-path$.
\end{definition}

So, given the simplicial complex $ \Delta $ associated with a discrete object, the construction of the gradient is given by a matching in its Hasse diagram.

\section{Greedy Discrete Gradient Field}

As we mentioned in the introduction, we would use the discrete gradient as a way to study a function $f$ evaluated at the vertices of a discrete object represented by $\Delta$. We also mention that there are several algorithms that propose methods of construction of the discrete gradient field. 

The vast majority of those algorithms are summed up in greedy matchings in weighted digraphs together with a choice of inducing weights of $f$ to the arcs of $H$. There are several ways to induce the weight of the function $f$ to the arcs (and thereby make $H$ a weighted digraph) \cite{Lewiner_1, King_Hnudson_Mramor, Gyulassy, Reininghaus_Lowen_Hotz, Paixao, lewiner2013critical}. In this work, we are following the choice of induction proposed in \cite{Paixao, lewiner2013critical}.

We need a total order on the vertices of the simplicial complex. Given a injective function $f$ defined at the vertices of the simplicial complex. Then, we define $f(\sigma)$ as the set of weights associated with the vertices of $\sigma$, that is, $f(\sigma)=\{f(v)|\forall v\in{\sigma}\}$.

We can extend the total order of the vertices to a total order of the simplicies with a lexicographic order.

\begin{definition}[Lexicographic Order]\label{def:lex_order}
If $\sigma_1 \subseteq \sigma_2$, then $\min f(\sigma_1) > \max f(\sigma_2) \setminus f(\sigma_1) \iff f(\sigma_1) > f(\sigma_2)$. If $\sigma_1 \not\subseteq \sigma_2$, then $\max f(\sigma_1) \Delta f(\sigma_2) \in f(\sigma_1) \iff f(\sigma_1) > f(\sigma_2)$, where $\setminus$ represents the subtraction operation in sets and $\Delta$ represents the symmetric difference.
\end{definition}

With this total order on the simplicies, we can now define the arc weights of the Hasse diagram.

\theoremstyle{definition}
\begin{definition}[Arc Weights] \label{def:arc_weight}
Given a function $f$ defined at the vertices of a simplicial complex, the weight induced by it in the arcs $\{\sigma,\tau\}$ of its Hasse diagram $H$ is given by $\overline{\{\sigma,\tau\}} = f(\tau\setminus\sigma)$.
\end{definition}

Observe that the weight of the arc $\{\sigma,\tau\}$, for simplicial complexes, is the value of the function $f$ at the vertex $v=\tau\setminus\sigma$. Since it is a simplicial complex, $\sigma$ is facet of $\tau$. Thus, the operation $\tau\setminus\sigma$ always results in a vertex $v$ where $v\in{\tau}$ and $v \notin{\sigma}$. Therefore we have that $\overline{\{\sigma,\tau\}} = f(v)$.

By an abuse of notation, we sometimes drop $f$, when the context is clear.

Once we have defined how to assign weights to the Hasse diagram arcs through a function evaluated at the vertices of a discrete object and how the greedy matching works, we can state in more detail how the discrete gradient is constructed (Algorithm~\ref{alg:discrete_gradient_field_construction_detail}).

\begin{algorithm}[H]

  \textbf{Input:} \text{Simplicial complex $\Delta$ and $f:V(\Delta)\rightarrow\mathbb{R}$} 
  \textbf{Output:} \text{Discrete gradient field $\mathcal{V}$}
  
  \caption{Discrete gradient field construction} \label{alg:discrete_gradient_field_construction_detail}
  
  \begin{algorithmic}[1]
  \State create Hasse diagram from $\Delta$
  \State induce weights on arcs of $H$ throught $f$
  \State  $\mathcal{V}$ = weighted matching algorithm \footnote{The algorithm should create a matching that hasn't closed trivial $\mathcal{V}-path$}
  \State \textbf{return} $\mathcal{V}$
  \end{algorithmic}

\end{algorithm}

\section{Geometric Faithfulness of Greedy Discrete Gradient Field}

First we will explain why the discrete vector field $\mathcal{V}$ associated with the greedy matching is in fact a discrete gradient field (which is equivalent to the smooth gradient in the sense of not having closed orbits). Next we will enunciate another geometric notion for the same: a notion of decreasing flow. The first result is already known in the literature \cite{Paixao, lewiner2013critical} for the greedy algorithm, in particular, and other constructs; we just simplify the proof. The second is a new result. It is also worth mentioning that all the results are valid for any dimension.

To show that the discrete vector field $\mathcal{V}$ associated with the greedy matching is in fact a discrete gradient field, it is the same as showing that such algorithm does not generate non-trivial closed $\mathcal{V}-path$  (Definition~\ref{def:discrete_gradient_field}). For this, we will propose some notations that will make the proof simple. The following notations and lemmas are our contributions to simplify evidence in the literature.

Considering a $\mathcal{V}-path$ denoted by $$\blacktriangleleft\sigma_0\tau_0\sigma_1\tau_1\dots\sigma_{n-1}\tau_{n-1}\sigma_n\blacktriangleright$$ we can write it as an alternating path $$P=\{\sigma_0,\tau_0\}\{\sigma_1,\tau_0\}\{\sigma_1,\tau_1\}\{\sigma_2,\tau_1\}\dots\{\sigma_{n-1},\tau_{n-1}\}\{\sigma_n,\tau_{n-1}\}.$$ In this way, according to the indices, we have that $\{\sigma_i,\tau_i\}\in{M}$ and $\{\sigma_{i-1},\tau_i\}\notin{M}$.

\begin{theorem} \label{theo:is_gradient}
The discrete vector field $\mathcal{V}$ associated with the greedy matching $M$ in the Hasse diagram $H$ is a discrete gradient field.
\end{theorem}

Following the theoretical assurances, the next theorem to be enunciated guarantees a notion of decreasing flow in a $\mathcal{V}-path$. Essentially, we want to show that from any simplex, by following the $\mathcal{V}-path$ to the critical simplex, then the latter is smaller than any other simplex in the path.

\begin{theorem}\label{theo:decreasing_flow}
In a discrete vector field $\mathcal{V}$ associated with the greedy matching $M$, for any $\mathcal{V}-path$ $\blacktriangleleft\sigma_0\tau_0\sigma_1\tau_1\dots\sigma_n\blacktriangleright$ we have that, if $\sigma_{n}$ is critical, then $f(\sigma_{n}) < f(\sigma_i)$ for all $0 \leq i \leq n-1$.
\end{theorem}

Still in the Theorem~\ref{theo:decreasing_flow}, it is worth mentioning that if $\sigma_n$ was not critical, then the path would not be a maximal alternating path (Definition \ref{def:maximal_alternating_path}). That way, we would not been able to use Theorem~\ref{thm:if_map_then_matched} in our proof. 

There is a guarantee that the critical simplex where $\mathcal{V}-path$ ends is the smallest of all simplexes in the path, however, we can not say the same for any intermediate path of simplexes in that $\mathcal{V}-path$.

Finally, to end this section, we will prove a sufficient condition for a simplex to be matched. First we define a set $H(\sigma)$ whose elements are vertices which provide the weights of all the arcs incident to $\sigma$.  

\begin{definition}\label{def:weights_of_h}
$H(\sigma) = \{v \in \Delta | \sigma \setminus v \in \Delta$ or $\sigma \cup v \in \Delta \}$ and $ h_f(\sigma) = argmin f(H(\sigma))$.
\end{definition}

The following lemma give a sufficient condition for a simplex to be matched.


\begin{theorem}[Steepest Descent]\label{thm:min_lemma_2}
If there exists $\tau \succ \sigma$ such that $\tau \setminus \sigma = h_f(\sigma)$, then $\sigma \rightarrow \tau$.
\end{theorem}

 A straighfoward corollary is a necessary condition for a critical simplex.
 
\begin{corollary}\label{cor:critical}
If $\sigma$ is critical, then $h_f(\sigma) \in \sigma$ and $h_f(\sigma \setminus h_f(\sigma)) \notin \sigma$.
\end{corollary}

\section{Discrete Smoothness} \label{sec:discrete_smoothing}

In this section, we introduce a concept of discrete smoothness for functions on a simplicial complex and then prove that for discrete smooth functions, the converse of the Steepest Descent Theorem (Theorem \ref{thm:min_lemma_2})  holds (Theorem \ref{thm:smooth}). In that case, we have a full characterization of the greedy Morse matching and its critical simplices. In Theorems \ref{thm:critical_above} and \ref{thm:critical_below}, we show that the dynamics around critical simplices are similar to their counterparts from the smooth theory. 


Our results on discrete smoothness generalize several results stated with barycentric subdivision~\cite{babson2005discrete,lewiner2005geometric,lewiner2013critical}, and we prove at the end of the Section \ref{sec:discrete_smoothing} that any function defined on a simplicial complex is discrete smooth after one barycentric subdivision.

We begin with the main definitions of this section.

\begin{definition}[Discrete Smoothness]\label{def:smooth}
A simplex $\sigma \in \Delta$ is discrete smooth, if and only if, if $h_f(\sigma) \in \sigma$, then for all $\tau \succ \sigma, h_f(\tau \setminus h_f(\sigma))=h_f(\sigma)$. A function $f$ is called \emph{discrete smooth} on $\Delta$, if for every $\sigma \in \Delta$, $\sigma$ is discrete smooth.
\end{definition}

Note that this definition depends on the function $f$ as well as the structure of the simplicial complex $\Delta$, and we write that the simplicial complex $\Delta$ is discrete smooth when the function defined on $\Delta$ is discrete smooth.
In this entire section we assume that $\Delta$ is discrete smooth.

Observe that the simplicial complexes in the right of Figure~\ref{fig:smoothness} are not discrete smooth.

\begin{figure}[h] 
    \centering
    \includegraphics[scale=0.70]{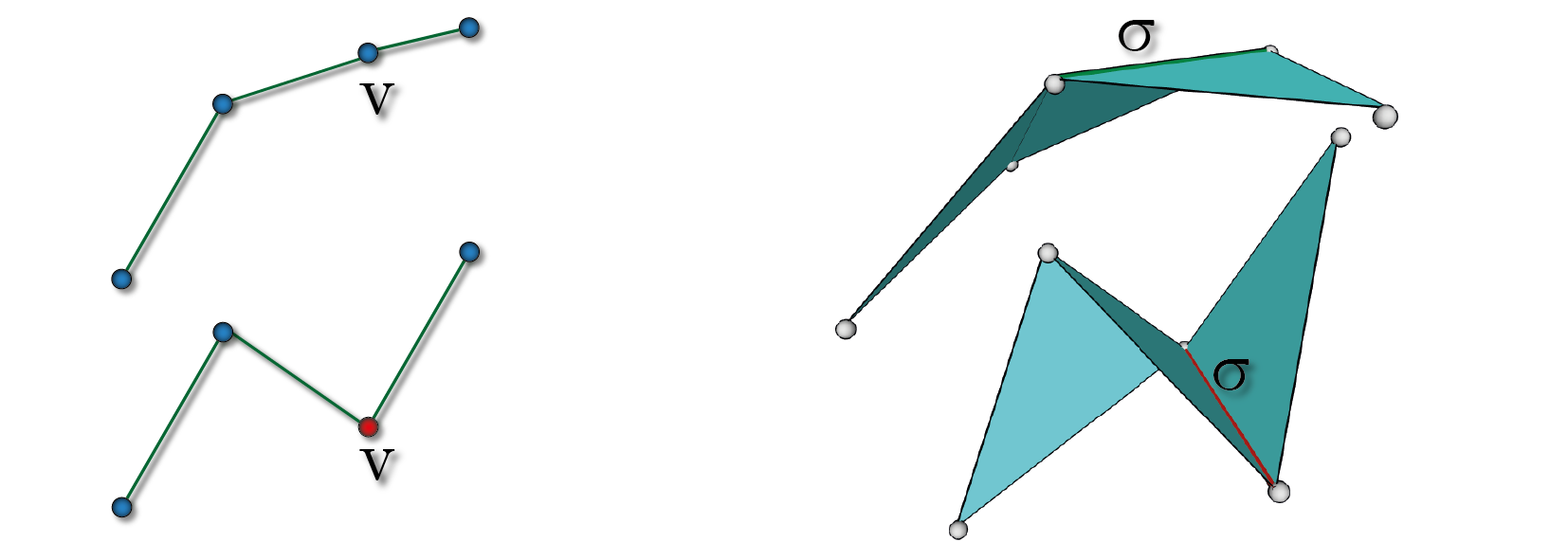}
    \caption{Examples of a smooth vertex $v$ (top left), a non-smooth vertex $v$ (bottom left), a smooth edge $\sigma$ (top right), and non-smooth edge $\sigma$ (bottom right). The function $f$ is the height function.}
    \label{fig:smoothness}
\end{figure}

With discrete smoothness we have a full characterization of the greedy Morse matching.

\begin{theorem}[Steepest Descent with Discrete Smoothness]\label{thm:smooth}

If $\sigma$ is smooth, then $\tau \succ \sigma$ such that $\tau \setminus \sigma = h_f(\sigma) \iff \sigma \rightarrow \tau$.
\end{theorem}

This characterization gives an important and extremely fast, linear and naturally parallel algorithm to compute greedy Morse matching (in the discrete smooth case): the algorithm only needs to check every $\sigma \in \Delta$ and $H(\sigma)$ to see if $\sigma$ is matched above. 

With the converse of the Steepest Descent Theorem, we can now prove the converse of Corollary \ref{cor:critical}, which follows directly by applying the  converse of the Steepest Descent Theorem.
\begin{corollary}\label{cor:critical_smooth}
$\sigma$ is critical if and only if $h_f(\sigma) \in \sigma$ and $h_f(\sigma \setminus h_f(\sigma)) \notin \sigma$.
\end{corollary}

%

\subsection{Geometrically Faithful Critical Simplicies and Paths}
Now that we have characterized the critical simplices, we want to understand the behavior of the greedy Morse matching in their neighborhood. We show that this behavior is similar to the local dynamics around the continuous critical points. In this entire section we assume that $\Delta$ is discrete smooth.

\begin{theorem}\label{thm:critical_above}
If $\sigma$ is critical, then all facets of $\sigma$ are matched above.
\end{theorem}

\begin{figure}[h] 
    \centering
    \includegraphics[scale=0.70]{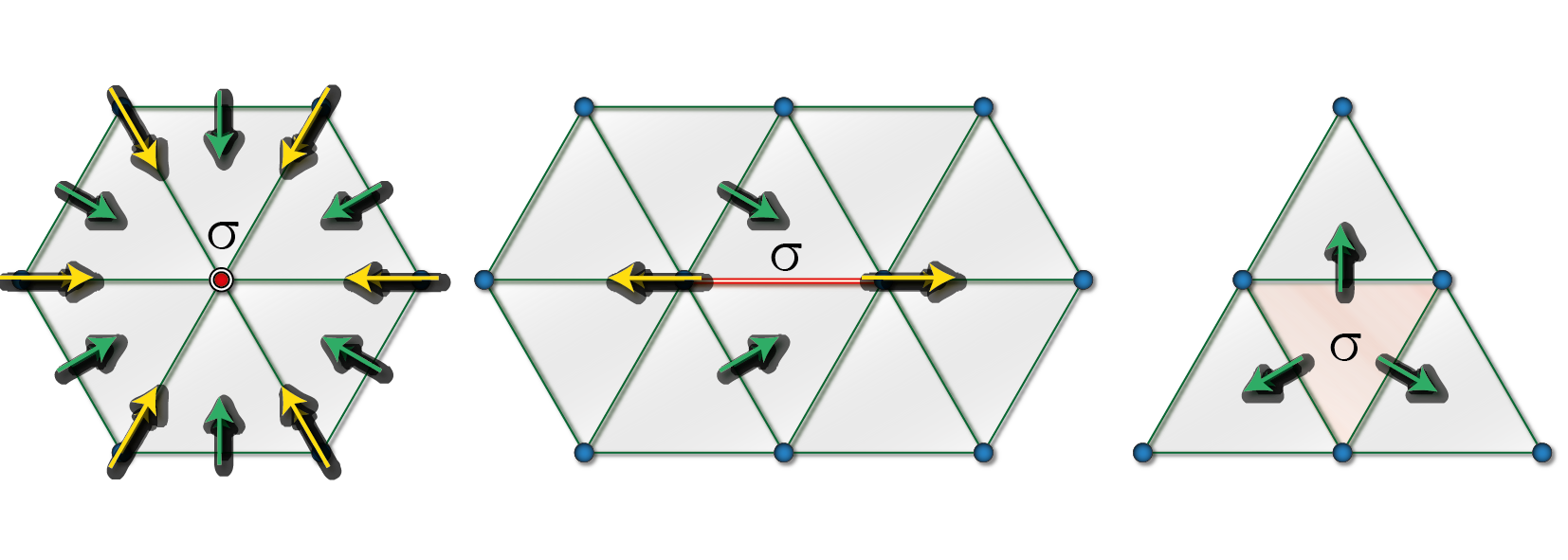}
    \caption{Results of Theorems \ref{thm:critical_above} and \ref{thm:critical_below} on a critical vertex (left), a critical edge (center), and a critical triangle (right). In a smooth complex, they are similar to minima (sinks), saddles, and maxima (sources).}
    \label{fig:nice_looking}
\end{figure}

\begin{theorem}\label{thm:critical_below}
If $\sigma$ is critical, then for all $\tau \succ \sigma, \tau  \setminus h_f(\sigma) \rightarrow \tau$.
\end{theorem}

The main characterization of the critical simplicies and its neighborhood follows easily from the two above Theorem.

\begin{theorem}[Geometrically Faithful Critical Simplicies]
$\sigma$ is critical if and only if for all $\tau \succ \sigma, \tau  \setminus h_f(\sigma) \rightarrow \tau$ and all facets of $\sigma$ are matched above.
\end{theorem}

Figure~\ref{fig:nice_looking} shows the Theorem above on surfaces (2-manifolds). In particular, Theorem \ref{thm:critical_below} states that all edges adjacent to a minimum points towards the minimum, and Theorem \ref{thm:critical_above} states that all facets of a maximum points outwards that maximum.

When the complex is smooth, we obtained a much stronger result than Theorem \ref{theo:decreasing_flow}, where we have truly decreasing paths.
\begin{theorem}\label{theo:decreasing_flow_smooth}
In a discrete vector field $\mathcal{V}$ associated with the greedy matching $M$, for any $\mathcal{V}-path$ $\blacktriangleleft\sigma_0\tau_0\sigma_1\tau_1\dots\sigma_n\blacktriangleright$ we have that, if $\sigma_{n}$ is critical, then $f(\sigma_{0}) > f(\sigma_{1}) > \cdots > f(\sigma_n)$.
\end{theorem}


%

\subsection{Discrete Smoothness of the Barycentric Subdivision}\label{sec:bary}

The goal of this section is to prove that, for any function $f$ defined on a simplicial complex $\Delta$, a function induced by $f$ on the barycentric subdivision of $\Delta$ is discrete smooth. This proves that the results of the previous sections apply to barycentric subdivision, generalizing previous results on similar greedy constructions of Morse matchings~\cite{babson2005discrete,lewiner2005geometric,lewiner2013critical}.

\begin{definition}[Barycentric Subdivision]\label{def:barycentric}
The barycentric subdivision $\Delta'$ of a simplicial complex $\Delta$ is a simplicial complex constructed as follows:
\begin{itemize}
    \item for every simplex $\sigma \in \Delta$, there is a vertex $b(\sigma) \in \Delta'$,
    \item for every sequence $\{\sigma_0,\dotsc, \sigma_p\} \subseteq \Delta$, such that $\sigma_0 \subseteq \sigma_1 \subseteq \dotsc \subseteq \sigma_p$, there is a $p$-simplex $\Sigma= b(\sigma_0),\dotsc, b(\sigma_p) \in \Delta'$.
\end{itemize}
\end{definition}


Since each vertex $b(\sigma)$ of the subdivision corresponds to an original simplex $\sigma$ of $\Delta$, we extend the function $f$ to the vertices in $\Delta'$ ordering vertex $b(\sigma)$ according to the lexicographic ordering of the vertices of $\sigma$  with a function $f'$ (see Figure~\ref{fig:barycentric}).


\begin{definition}[Extension of $f$]\label{def:extension}
The extension $f'$ of function $f$ on the vertices of $\Delta'$ is defined by $f'(b(\sigma)) := f(\sigma)$. The values of $f'$ are totally ordered using lexicographic order from Definition \ref{def:lex_order}. Since $\Delta$ is a simplicial complex and $f$ is injective, then $f'$ is injective.
\end{definition}

To illustrate the ordering, consider $\Delta$ as a single triangle $abc$ and its faces, with $f(a) > f(b) > f(c)$. The lexicographic ordering orders the faces of $\Delta$ as: $f(a)> f(ab) > f(abc)   > f(ac) >  f(b) > f(bc) > f(c)$.

\begin{figure}[h] 
    \centering
    \includegraphics[scale=0.70]{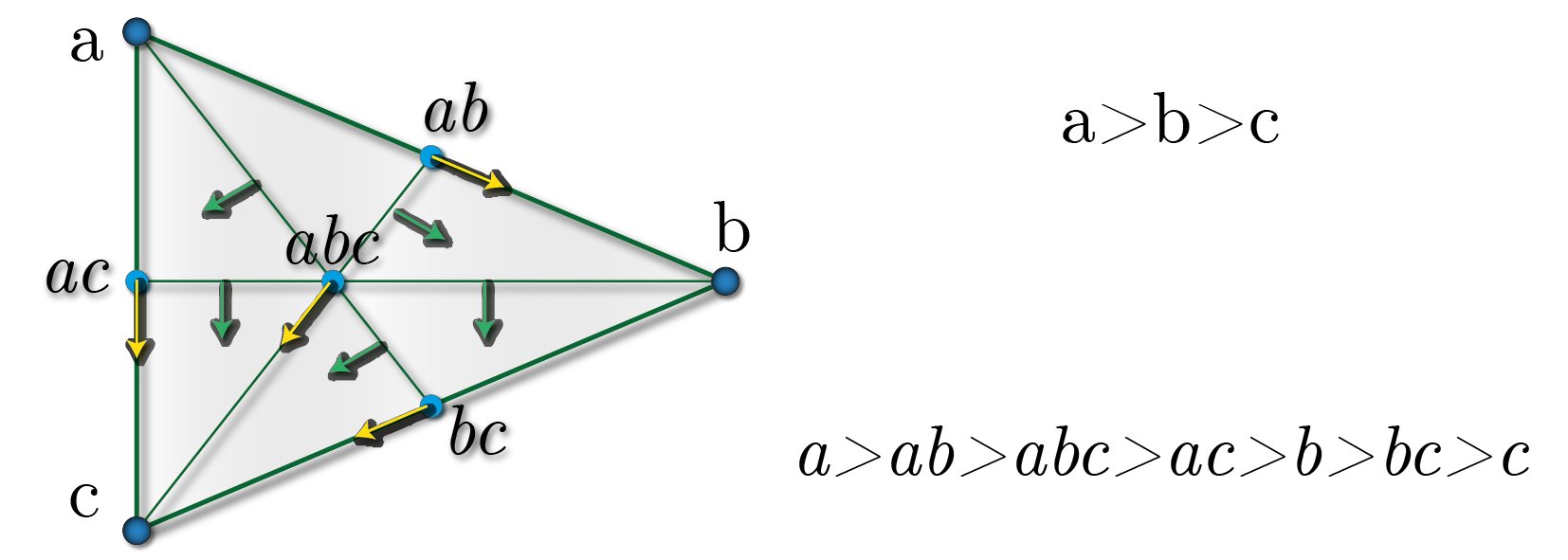}
    \caption{Barycentric subdivision of simplex $abc$, indicating the pairs that are always matched in the greedy Morse matching.}
    \label{fig:barycentric}
\end{figure}


\begin{theorem}[Discrete Smoothness of the barycentric subdivision]\label{thm:bary_smooth}
Given a function $f$ on $K$, the extension of the function $f$ on the barycentric subvision of $K$ is smooth.
\end{theorem}

\section{Greedy Matchings on Polyhedral Complexes}
\begin{definition}

In this section, we extend the greedy matchings beyond simplicial complexes.

A polyhedral complex in $\mathbb{E}^m$ is a set, $K$, consisting of a (finite) set of convex polytopes in $\mathbb{E}^m$ satisfying the
following conditions:

\begin{enumerate}
    \item Every face of a polytope in $K$ also belongs to $K$.
    \item For any two polytopes $\sigma_1$ and $\sigma_2 \in K$, if $\sigma_1 \cap \sigma_2 \neq \emptyset$,
then $\sigma_1 \cap \sigma_2 $ is a common face of both $\sigma_1$ and $\sigma_2$.
\end{enumerate}

\end{definition}


To deal with polyhedral complexes, we define a modified Hasse diagram where the function $f$ is taken under consideration.

\theoremstyle{definition}
\begin{definition} [Modified Hasse Diagram] \label{def:modified_hasse_diagram}
Given a simplicial complex $\Delta$, its Hasse diagram $H'$ is the directed graph $D$ where the set of nodes is the set of simplexes of $\Delta$ and the set of arcs in $H'$ is composed of all pairs such that $\{\sigma,\tau\} \in H'$ if and only if $\sigma\prec\tau$ and $f(\tau) < f(\sigma)$.
\end{definition}

Observe that for simplicial complexes, this definition does not reduces to Definition \ref{def:hasse_diagram}. This is not in fact a generalization of previous sections but an alternative when dealing with polyhedral complexes. Please note that, since arcs of the original Hasse diagram are missing, there might be two critical adjacent cells, which was impossible in the previous definitions in simplicial complex. 

If we consider this modified Hasse diagram, we can show that the greedy matching algorithm constructs a discrete gradient field with decreasing paths such as Theorem \ref{theo:decreasing_flow_smooth} and its follows that the matching produced is in fact a discrete gradient field.

\begin{theorem}\label{theo:decreasing_flow_modified}
In a discrete vector field $\mathcal{V}$ associated with a matching $M$ in the Modified Hasse diagram $H'$, for any $\mathcal{V}-path$ $\blacktriangleleft\sigma_0\tau_0\sigma_1\tau_1\dots\sigma_n\blacktriangleright$ we have that $f(\sigma_{0}) > f(\sigma_{1}) > \cdots > f(\sigma_n)$.
\end{theorem}

\begin{corollary}\label{theo:modified_is_gradient}
The discrete vector field $\mathcal{V}$ associated with a matching $M$ in the Modified Hasse diagram $H$ is a discrete gradient field.
\end{corollary}

\section{Application to CAT(0) Cubical Complexes} 
\label{sec:applications}
In this section, we will show that CAT(0) cubical complexes are collapsible (a result established in \cite{adiprasito2011metric} using convexity),  by applying the greedy matching on the modified Hasse diagram from the previous section. We will use a fully combinatorial description of the CAT(0) cubical complex recently developed in \cite{ardila2012geodesics}. We include their combinatorial description verbatim from \cite{ardila2012geodesics} for completeness. Please see the reference for examples and details.

Recall that a poset $P$ is locally finite if every interval $[i, j] = \{k \in P : i \leq j \leq k\}$ is finite, and it has finite width if every antichain (set of pairwise incomparable elements) is finite.

\theoremstyle{definition}
\begin{definition}  \cite{ardila2012geodesics} \label{def:PIP}
A poset with inconsistent pairs is a locally finite poset $P$ of finite width,
together with a collection of inconsistent pairs ${p, q}$, such that:
\begin{enumerate}

 \item If $p$ and $q$ are inconsistent, then there is no $r$ such that $r \geq p$ and $r \geq q$.
\item If $p$ and $q$ are inconsistent and $p'\geq p, q'\geq q$, then $p'$ and $q'$ are inconsistent.
\end{enumerate}

\end{definition}

\begin{figure}[ht] 
    \centering
    \includegraphics[scale=0.58]{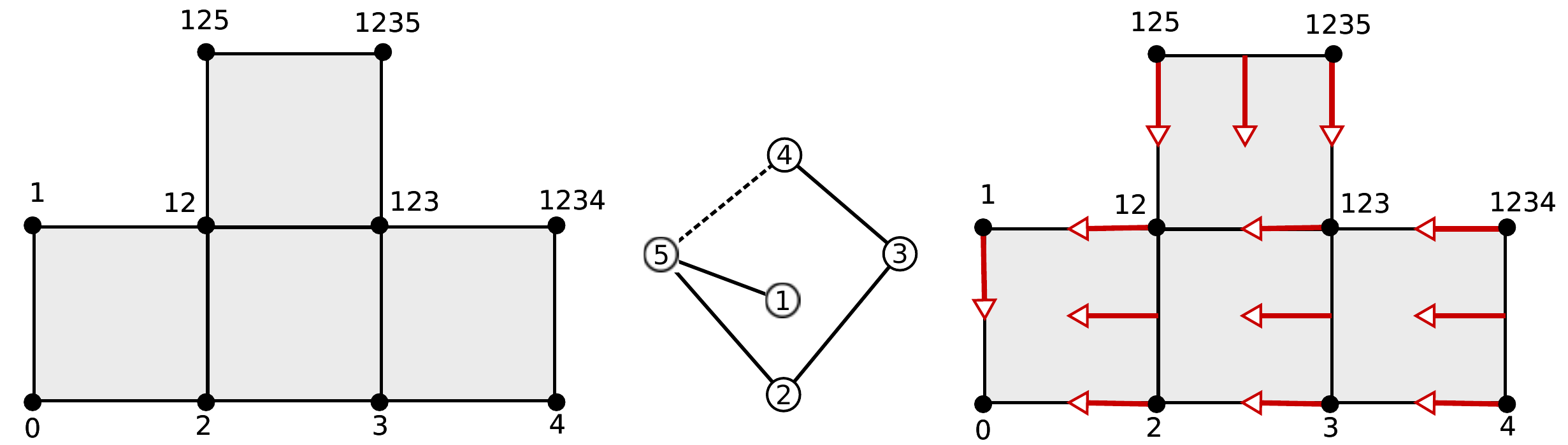}
    \caption{A rooted CAT(0) complex (left), built with Definition \ref{def:cat0_complex} from a Poset with Inconsistent Pairs (center), and the greedy matching defined in Theorem \ref{thm:cat0}.}
    \label{fig:cat0pip}
\end{figure}

In particular, notice that any two inconsistent elements must be incomparable.
A diagram of a poset with inconsistent pairs is obtained by drawing the poset, and connecting each minimal inconsistent pair with a dotted line. An inconsistent pair
$\{p, q\}$ is minimal if there is no other inconsistent pair ${p_0, q_0}$ with $p_0 \leq p$ and $q_0 \leq q$ (See Figure \ref{fig:cat0pip}).

\theoremstyle{definition} 
\begin{definition} \cite{ardila2012geodesics} \label{def:cat0_complex}
If $P$ is a poset with inconsistent pairs, we construct the cube complex
of $P$, which we denote $X_P$ . The vertices of $X_P$ are identified with the consistent order
ideals of $P$. There will be a cube $C(I, M)$ for each pair $(I, M)$ of a consistent order ideal
$I$ and a subset $M \subseteq I_{max}$, where $I_{max}$ is the set of maximal elements of $I$. This cube has dimension $|M|$, and its vertices are obtained by removing from $I$ the $2^{|M|}$ possible subsets of $M$. These cubes are naturally glued along their faces according to their labels.
\end{definition}

\begin{theorem}[Combinatorial description of CAT(0) cubical complexes.] \label{theo:combinatorial_cat0} \cite{ardila2012geodesics} 
There is a bijection between posets with inconsistent pairs and rooted CAT(0) cube complexes, given by the map $P \mapsto X_P$ .
\end{theorem}

The next theorem was already proven in \cite{adiprasito2011metric} but here we give fully combinatorial proof. We will define a order on the cubes $C(I,M)$ of $X_P$, construct a discrete gradient field with the greedy matching on the modified Hasse diagram, and finally we will show that there is only one critical vertex in the discrete gradient field. Therefore the complex is collapsible.

\begin{theorem}\label{thm:cat0}
CAT(0) cube complexes are collapsible.
\end{theorem}






\bibliographystyle{plain}
\bibliography{main.bib}

\appendix
\section{Appendix}

\subsection{Proof of Theorem \ref{thm:if_map_then_matched}}

\begin{lemma} \label{lem:H_in_Ga}
If $H$ is a subgraph of $G$ such that $\forall v \in V(H)$ we have $\overline{v} \geq a$, then $H$ is a subgraph of $G_a$.
\end{lemma}

\begin{proof}
If $H$ is a subgraph of $G$, then $V(H) \subseteq{V(G)}$ and $E(H) \subseteq {E(G)}$. In addition, if $\forall v \in V(H)$ we have $\overline{v} \geq a$, then $v \in {V_a}$. Thus, $V(H) \subseteq{V_a}$. Now consider an edge $e = \{x, y \} \in{E(H)}$. It is easy to verify that $\overline{x} \geq{a}$ and $\overline{y} \geq{a}$. This means that $x \in{V_a}$ and $e \in{V_a}$. Since $G_a$ is an induced subgraph of $G$ for $V_a$, it follows that $e = \{x, y \} \in{G_a}$, and therefore $E(H) \subseteq{E_a}$. Therefore, $H$ is a subgraph of $G_a$.
\end{proof}

\begin{lemma}\label{lem:v_geq_e}
Let $e \in {E_{min}(P)}$. If $P$ is a maximal alternating path, then $\forall v \in P$ we have $\overline {v} \geq \overline{e} $.
\end{lemma}

\begin{proof}
Consider $e \in{E_{min}(P)}$ as one of the edges of minimum weight in $P$. If $P$ is a maximal alternating path, then $\forall v \in P$ we have that $v$ is saturated by an edge of $P$ or $v$ is unsaturated (Definition~\ref{def:maximal_alternating_path}). Thus, $\forall v \in P$ follows that $\overline{v} \geq \overline{e}$ or $\overline{v} = \infty $, respectively.
\end{proof}

\bigskip
\noindent {\bf Proof of Theorem \ref{thm:if_map_then_matched}}
\begin{proof}
Let $e$ be an edge belonging to the set of edges of minimum weight in $P$, that is, $e \in{E_{min}(P)}$. If $P$ is a maximal alternating path, then $\forall v \in P$ we have $ \overline{v} \geq \overline{e}$ by Lemma \ref{lem:v_geq_e}. Furthermore, by Lemma ~\ref{lem:H_in_Ga} we have that $P$ is a subgraph of $G_{\overline{e}}$.
In particular $e \in P \subset G_{\overline{e}}(E)$. Finally we have that $e \in M$, by the Lemma~\ref{lem:is_matched}. It follows that $E_{min}(P)\subseteq{M}$.
\end{proof}


\subsection{Proofs of Theorems \ref{theo:is_gradient} and \ref{theo:decreasing_flow}}

As $\sigma_i$ and $\sigma_{i-1}$ are facets of $\tau_i$ in the $\mathcal{V}-path$, we will call as $v_i=\tau_i\setminus \sigma_i$ and $w_i=\tau_{i}\setminus \sigma_{i-1}$ two vertices. We can simplify those equalities checking that $\sigma_{i-1} \setminus \sigma_{i} = v_i$ and $\sigma_{i-1} \setminus \sigma_{i} = w_i$. This indicates that, roughly speaking, for any intermediate $\mathcal{V}-path$ defined as $\blacktriangleleft\sigma_{i-1}\tau_i\sigma_i \blacktriangleright$, we have that $\sigma_i$ obtained a vertex $v_i$ and lost a vertex $w_i$ in relation to the previous simplex $\sigma_{i-1}$. Thus, we will name as $V_n=\{v_0,v_1,v_2,\dots,v_{n-1}\}$ and $W_n=\{w_0,w_1,w_2,\dots,w_{n-1}\}$ two sets of vertices. The first one represents a set of vertices that was obtained throughout the path and the second one is a set of vertices that was lost throughout the path.

This means that, $f(V_n)$ is the set of weights of the matched edges and $f(W_n)$ is the set of weights of the unmatched edges in a $\mathcal{V}-path$.

The relations of those sets, $V$ and $W$, with the simplexes along a $\mathcal{V}-path$ are intuitive and consist of useful tools to prove results related with the discrete gradient field.

\begin{lemma}\label{lem:vpath_as_alternate_path_1} $V_n \setminus W_n \subseteq \sigma_n \setminus \sigma_0$ and $W_n \setminus V_n \subseteq \sigma_0 \setminus \sigma_n$.
\end{lemma}

\begin{proof}

As the proof is symmetrical, consider only one of the relations. If a vertex $y \in V_n \setminus W_n$ then $y \in V$ and $y \notin W$. Consider the two next statements.

\begin{enumerate}

\item If $y \in V_n$ then there are at least a pair of simplexes $\sigma_i$ and $\sigma_{i+1}$ for $0 \leq i \leq n-1$ in $P$ such that $\sigma_{i+1} \setminus \sigma_{i} = y$. This implies that $y \notin \sigma_{i}$ and $y \in \sigma_{i+1}$.

\item Similarly, if $y \notin W_n$ so there is no pair of simplexes $\sigma_{j}$ and $\sigma_{j+1}$ for $0 \leq j \leq n-1$ in $P$ such that $\sigma_{j} \setminus \sigma_{j+1} = y$. This means that if $y \in \sigma_j \rightarrow y \in \sigma_{j+1}$.

\end{enumerate}

Combining the two statements, we conclude that $y \notin \sigma_0$ and $y \in \sigma_n$.
\end{proof}

\begin{corollary}\label{cor:vpath_as_alternate_path_3} For a non-trivial closed $\mathcal{V}-path$, we have that $V_n=W_n$.
\end{corollary}

\begin{proof}

If a $\mathcal{V}-path$ denoted by $\blacktriangleleft\sigma_0\tau_0\sigma_1\tau_1\dots\sigma_n\blacktriangleright$ is non-trivial closed, then $n > 0$ and $\sigma_0 = \sigma_n$. Thus, we have that $V_n \setminus W_n \subseteq \sigma_n \setminus \sigma_0 = \emptyset$ and $W_n \setminus V_n \subseteq \sigma_0 \setminus \sigma_n = \emptyset$. This means that $V_n=W_n$

\end{proof}

The corollary shows that the set of all the lost vertices and the set of all the obtained vertices in a non-trivial and closed $\mathcal{V}-path$ are the same. This is not a surprise since $\sigma_0 = \sigma_n$.

\bigskip

\noindent {\bf Proof of Theorem \ref{theo:is_gradient}}
\begin{proof}
Suppose by contradiction that $\mathcal{V}$ is not a discrete gradient field. Therefore, by the Definition~\ref{def:discrete_gradient_field}, there exists a closed non-trivial $\mathcal{V}-path$ denoted by $$P=\{\sigma_0,\tau_0\}\{\sigma_1,\tau_0\}\{\sigma_1,\tau_1\}\{\sigma_2,\tau_1\}\dots\{\sigma_{n-1},\tau_{n-1}\}\{\sigma_{n},\tau_{n-1}\}$$ where $\sigma_0 = \sigma_n$. Because $P$ is an alternating cycle, it is also a maximal alternating path, then $E_{min}(P) \subset M$ (Theorem~\ref{thm:if_map_then_matched}). Since $f(V_n)$ is the set of weights of the matched edges in $P$ and $f(W_n)$ is the set of weights of the unmatched edges in $P$, we have that $\min f(V_n \cup W_n) \in f(V_n \setminus W_n)$ which also belongs to $f(\sigma_n \setminus \sigma_0)$ (Lemma~\ref{lem:vpath_as_alternate_path_1}). As $\sigma_0 = \sigma_n$, then $\sigma_n \setminus \sigma_0 = \emptyset$. Contradiction.
\end{proof}

\begin{lemma}\label{lem:vpath_as_alternate_path_2} 
$\sigma_0 \setminus \sigma_n \subseteq W_n$ and $\sigma_n \setminus \sigma_0 \subseteq V_n$
\end{lemma}

\begin{proof}

If a vertex $y \in \sigma_0 \setminus \sigma_n$ then $y \in \sigma_0$ and $y \notin \sigma_n$. This means that exist a pair of simplexes $\sigma_i$ and $\sigma_{i+1}$ where $0 \leq i \leq n-1$ such that $y \in \sigma_i$ and $y \notin \sigma_{i+1}$. Thus, $y \in W_n$. Similarly, if a vertex $y \in \sigma_n \setminus \sigma_0$ then $y \in \sigma_n$ and $y \notin \sigma_0$. This means that exist a pair of simplexes $\sigma_{i+1}$ and $\sigma_{i}$ where $0 \leq i \leq n-1$ such that $y \in \sigma_{i+1}$ and $y \notin \sigma_{i}$. Thus, $y \in V_n$.

\end{proof}

This last lemma simply indicates that if a vertex is not in the initial simplex of the $\mathcal{V}-path$ but it is in the final simplex, then this vertex is in the set of obtained vertices. Similarly, if a vertex is in the initial simplex but is not in the final simplex, then this vertex is in the set of lost vertices.

\bigskip
\noindent {\bf Proof of Theorem \ref{theo:decreasing_flow}}
\begin{proof}
Consider a $\mathcal{V}-path$ denoted by $$P=\{\sigma_0,\tau_0\}\{\sigma_1,\tau_0\}\{\sigma_1,\tau_1\}\{\sigma_2,\tau_1\}\dots\{\sigma_{n-1},\tau_{n-1}\}\{\sigma_{n},\tau_{n-1}\}$$ with $\sigma_n$ critical. Because $\sigma_n$ is critical, then $P$ is a maximal alternating path, then $E_{min}(P) \subset M$ (Theorem~\ref{thm:if_map_then_matched}). Since $f(V_n)$ is the set of weights of the matched edges in $P$ and $f(W_n)$ is the set of weights of the unmatched edges in $P$, we have that $\min f(V_n \cup W_n) \in f(V_n \setminus W_n)$ which also belongs to $f(\sigma_n \setminus \sigma_0)$ (Lemma~\ref{lem:vpath_as_alternate_path_1}). As $\min f(V_n \cup W_n) \leq \min f((\sigma_0 \setminus \sigma_n) \cup (\sigma_n \setminus \sigma_0) )$ (Lemma \ref{lem:vpath_as_alternate_path_2}, then $\min f((\sigma_0 \setminus \sigma_n) \cup (\sigma_n \setminus \sigma_0) ) \in f(\sigma_n \setminus \sigma_0)$, this means that $f(\sigma_n) < f(\sigma_0)$.
\end{proof}

\subsection{Proof of Theorem \ref{thm:min_lemma_2}}

\begin{lemma}\label{lem:H_weights}
If $\sigma \subseteq \tau$, then $H(\tau) \subseteq H(\sigma)$ and $h_f(\sigma) \leq h_f(\tau)$. Also $\overline{\sigma} \geq h_f(\sigma)$.
\end{lemma}

\begin{proof}
Trivial, just set operations.
\end{proof}

\bigskip

\noindent {\bf Proof of Theorem \ref{thm:min_lemma_2}}
\begin{proof}
With Lemma \ref{lem:H_weights} and Definition \ref{def:weights_of_h}, we have that $\overline{\{\sigma,\tau\}}= f(\tau \setminus \sigma) = h_f(\sigma) \leq \overline{\sigma}$ and $\overline{\{\sigma,\tau\}}= f(\tau \setminus \sigma) = h_f(\tau) \leq \overline{\tau}$. By Lemma \ref{thm:if_map_then_matched} we can conclude that, $\sigma \rightarrow \tau$.
\end{proof}


\subsection{Proofs of Theorems \ref{thm:smooth}, \ref{thm:critical_above}, and \ref{thm:critical_below}}

\bigskip
\noindent {\bf Proof of Theorem \ref{thm:smooth}}
\begin{proof} $\implies$ Theorem \ref{thm:min_lemma_2}.

$\impliedby$ First suppose $h_f(\sigma) \in \sigma$, since $\sigma$ is smooth, then by Definition \ref{def:smooth}, $\tau \setminus h_f(\sigma)=h_f(\sigma) \notin \tau \setminus h_f(\sigma)$. By Theorem \ref{thm:min_lemma_2}, $\tau \setminus h_f(\sigma) \rightarrow \tau \setminus h_f(\sigma) \cup h_f(\sigma)=\tau$. That is a contradiction, since $\tau$ is already matched with $\sigma$ and $\tau \setminus h_f(\sigma) \neq \sigma$, since $h_f(\sigma) \in \sigma$. Therefore $h_f(\sigma) \notin \sigma$ and it follows that  there exists$\tau' \succ \sigma$ such that $\tau' \setminus \sigma = h_f(\sigma)$. By Theorem \ref{thm:min_lemma_2}, $\sigma \rightarrow \tau'$. Since it is a matching $\tau'=\tau$.
\end{proof}

\bigskip

\noindent {\bf Proof of Theorem \ref{thm:critical_above}}
\begin{proof} Suppose there exists a facet $\rho \prec \sigma$ such that $\rho$ is not matched above. Since $\rho$ is not critical by Corollary \ref{cor:critical_smooth}, $\rho$ is matched below with some $\theta$. By Theorem \ref{thm:smooth}$, \rho \setminus \theta=h_f(\theta)$. Also since $\theta \subseteq \rho \subseteq \sigma$, $\theta \subseteq \sigma \setminus (\rho \setminus \theta)$. Therefore, since $\rho \setminus \theta =h_f(\theta)$, by Theorem \ref{thm:smooth} $\sigma \setminus (\rho \setminus \theta) \rightarrow \sigma$, which is a contradiction since $\sigma$ is critical.
\end{proof}

\bigskip

\noindent {\bf Proof of Theorem \ref{thm:critical_below}}
\begin{proof}
Since $\sigma$ is critical, then $h_f(\sigma) \in \sigma$ by Theorem \ref{thm:min_lemma_2}. Since $\sigma$ is smooth, by Definition \ref{def:smooth} for all $\tau \succ \sigma, h_f(\tau \setminus h_f(\sigma))=h_f(\sigma)=\tau \setminus (\tau \setminus h_f(\sigma))$. Therefore, by Theorem \ref{thm:smooth}  $\tau  \setminus h_f(\sigma) \rightarrow \tau$.

\end{proof}

\subsection{Proof of Theorem \ref{thm:bary_smooth}}

\begin{definition}
If $\Sigma= (\sigma_0, \sigma_1, \ldots, \sigma_p)$, then let $\Sigma_i= (\sigma_0, \sigma_1, \ldots, \sigma_i)$ and $p_i=\min H(\Sigma_i) \setminus \Sigma_i$.
\end{definition}

\begin{lemma}\label{lem:p_lemma}
$p_i= \sigma_{i-1} \cup \min \sigma_{i} \setminus \sigma_{i-1}$.
\end{lemma}
\begin{proof}
Let $x = \sigma_{i-1} \cup \min \sigma_{i} \setminus \sigma_{i}$. Since $\sigma_{i} \setminus \sigma_{i-1} \neq \emptyset$, then $\sigma_{i-1} \subset x$. It is easy to see that $\sigma_{i-1} \subset x \subseteq \sigma_{i}$. Therefore $x\in H(\Sigma_i) \setminus \Sigma_i$ and it follows that $p_i \leq x$. In addition, $p_i \setminus x \subseteq \sigma_{i} \setminus \sigma_{i-1}$ and $x \setminus p_i=\min \sigma_{i} \setminus \sigma_{i-1}$. Therefore, by the lex order, $p_i \geq x$.
\end{proof}

\bigskip

\noindent {\bf Proof of Theorem \ref{thm:bary_smooth}}
\begin{proof}
Let $\sigma_m=h_f(\Sigma) \in \Sigma$ and $\sigma''= h_f(\Sigma')$ where $\Sigma'=T \setminus h_f(\Sigma)$, for some $T \succ \Sigma$. To prove discrete smoothness, we must show that $\sigma_m=\sigma''$. We will first show that for all $i\leq m, \sigma_i=p_i$ (1), $\dim(\sigma_i)=i$ (2), and $\Sigma_i=\Sigma'_i$ (3):

\begin{enumerate}
    \item $\sigma_i=p_i$. Suppose there exists $i \leq m$ such that $\sigma_i \neq p_i$. Therefore we know, by Lemma \ref{lem:p_lemma}, that $\sigma_{i-1} \subset p_i \subset \sigma_i$. It follows that $p_i \notin \Sigma$, but we know that $p_i \in H(\Sigma)$. Therefore $p_i \in H(\Sigma) \setminus \Sigma$ and $p_i \neq \sigma_m \in \Sigma$. Finally we have that $p_i<\sigma_m $.
    On the other hand, since $i \leq m$, we know that $\sigma_m \in H(\Sigma_i) \setminus \Sigma_i$. Therefore $p_i \leq \sigma_m$, a contradiction. 
    \item $\dim(\sigma_i)=i$. It follows from Lemma \ref{lem:p_lemma} and Item 1 that $1=\dim(p_0)=\dim(\sigma_0)$ and $\dim(\sigma_i)=\dim(p_i)=\dim(\sigma_{i-1})+1$. By induction, $\dim(\sigma_i)=i$.
\end{enumerate}

 Now suppose $\dim(\sigma'') \geq m$, then $\sigma'' \notin \Sigma'_m$. Therefore $\sigma'' \in H(\Sigma_m)\setminus \Sigma_m$. Then $p'_m \leq \sigma''$. Since $\sigma_m \in H(\Sigma')$, it is easy to see that $\sigma_m=p_m$, and $\Sigma_m=\Sigma'_m$. We know that $\sigma'' \leq \sigma_m = p_m = p'_m \leq \sigma''$. Finally we have $h_f(\tau \setminus h_f(\sigma))=h_f(\Sigma')=\sigma''=\sigma_m=h_f(\Sigma)$.
\end{proof}

\subsection{Proofs of Theorems \ref{theo:decreasing_flow_modified} and \ref{theo:modified_is_gradient}}

\bigskip

\noindent {\bf Proof of Theorem \ref{theo:decreasing_flow_modified}}

\begin{proof}
Consider a $\mathcal{V}-path$ denoted by $\blacktriangleleft\sigma_0\tau_0\sigma_1\tau_1\dots\sigma_n\blacktriangleright$. Since the arcs $\{\tau_i \ \sigma_i \}$ are matched and, therefore, exist in the Hasse Diagram, by Definition~\ref{def:modified_hasse_diagram} we have that $\tau_i < \sigma_i$. It follows from the lexicographic order (Definition~\ref{def:lex_order}) that $\sigma_{i+1} < \tau_i$ for all $0 \leq i \leq n-1$ . This implies that $\sigma_{i+1} < \tau_i < \sigma_i$. By simple induction, is guaranteed that the $\mathcal{V}-path$ is strictly decreasing.
\end{proof}

\bigskip
\noindent {\bf Proof of Theorem
\ref{theo:modified_is_gradient}}

\begin{proof}
Suppose by contradiction that $\mathcal{V}$ is not a discrete gradient field. Therefore, by the Definition~\ref{def:discrete_gradient_field}, there exists a closed non-trivial $\mathcal{V}-path$ denoted by $\blacktriangleleft\sigma_0\tau_0\sigma_1\tau_1\dots\sigma_n\blacktriangleright$ where $\sigma_0 = \sigma_n$. As the sequence of simplexes are strictly decreasing, $\sigma_0 = \sigma_n$ represents a contradiction.
\end{proof}

\subsection{Proof of Theorem \ref{thm:cat0}}
The following lemma is taken verbatim from \cite{ardila2012geodesics} to help prove Theorem \ref{thm:cat0}.

\begin{lemma}\label{lem:helper} Let $J$ be a consistent order ideal and let $N \subseteq J_{max}$. The faces of the cube
$C(J, N)$ in the cubical complex $X_P$ are the $3^{|N|}$
cubes $C(J \setminus N_1, N \setminus N_1 \setminus N_2)$, where $N_1$
and $N_2$ are disjoint subsets of N. The maximal cubes in $X_P$ correspond to the maximal consistent antichains $A$ of $P$.
\end{lemma}

We can order the vertices of $X_P$, which by Definition \ref{def:cat0_complex} are the consistent order ideals of $P$, with any linear ordering of the elements of $P$. Now we can use a shortlex $<_{sl}$ ordering to order all the cubes such as: $(I,M) < (I',M')$ if and only if $I <_{sl} I'$ or $(I = I'$ and $M >_{sl} M')$. Note the opposite signs when comparing $I$ and $I'$ or $M$ and $M'$.

In shortlex ordering, the sequences are primarily sorted by cardinality (length) with the shortest sequences first, and sequences of the same length are sorted into lexicographical order from Definition \ref{def:lex_order}.

\bigskip

\noindent {\bf Proof of Theorem \ref{thm:cat0}}

\begin{proof}
Apply the greedy matching $M$ with the cubes ordered by the order above. If $I \neq \emptyset$, then define $p=\max I_{max}$. We will show that if $p\in M$, then $(I,M \setminus p) \rightarrow (I,M)$ (See Figure \ref{fig:cat0pip}). With this we can conclude that every $(I,M)$ is matched except when $I=\emptyset$, since if $I \neq \emptyset$, then $(I,M\setminus p) \rightarrow (I,M)$, if $p\in M$ or $(I,M) \rightarrow (I,M \cup p)$ if $p \notin M$. Therefore $X_P$ is collapsible since it has one critical vertex, the empty set ideal.

Let $\sigma=(I,M \setminus p)$ and $\tau=(I,M)$.

Since $|M|>|M\setminus p|$, then $M>_{sl} M\setminus p$. Therefore $\tau = (I,M) < (I,M \setminus p) = \sigma$. Therefore  $\{\sigma,\tau\}\in H$, by Definition \ref{def:modified_hasse_diagram}.

  We will compare the weight of $\{\sigma, \tau\}$ against the weights of every other possible pair that $\sigma$ or $\tau$ can be matched. The proof is simply using the order above with Lemma \ref{lem:helper} and Theorem \ref{thm:if_map_then_matched}.

\begin{enumerate}
\item Suppose $\tau \rightarrow \tau'$. Since $\tau \prec \tau'$, then $\tau' \setminus \tau = (I \cup v, M)$ for some $v \in P$. We have that $\overline{\{\tau,\tau'\}} =\tau'\setminus\tau=(I \cup v, M)>(I\setminus p,M \setminus p)=\tau \setminus \sigma= \overline{\{\sigma,\tau\}}$.

\item Suppose $\tau' \rightarrow \tau$. Since $\tau' \prec \tau$, then $\tau \setminus \tau' = (I, M\setminus v)$ or $\tau'=(I\setminus v,M \setminus v)$ for some $v \in P$. We have that $\overline{\{\tau',\tau\}} =\tau\setminus\tau'=(I, M \setminus v)>(I\setminus p,M \setminus p)=\overline{\{\sigma,\tau\}}$ or $\overline{\{\tau',\tau\}} =\tau\setminus\tau'=(I \setminus v, M \setminus v)>(I\setminus p,M \setminus p)=\tau \setminus \sigma=\overline{\{\sigma,\tau\}}$, since $p=\max I_{max}$.

\item Suppose $\sigma \rightarrow \sigma'$. Since $\sigma \prec \sigma'$, then $\sigma' \setminus \sigma = (I \cup v, M\setminus p)$ for some $v \in P$. We have that $\overline{\{\sigma,\sigma'\}} =\sigma'\setminus\sigma=(I \cup v, M)>(I\setminus p,M \setminus p)=\tau \setminus \sigma= \overline{\{\sigma,\tau\}}$. 
    
\item Suppose $\sigma' \rightarrow \sigma$. Since $\sigma' \prec \sigma$, then $\sigma \setminus \sigma' = (I, M\setminus p\setminus v)$ some $v \in P$ or $\sigma'=(I\setminus v,M\setminus p \setminus v)$ for some $v \in P$ such that $v \neq p$. We have that $\overline{\{\sigma',\sigma\}} =\sigma\setminus\sigma'=(I, M\setminus p \setminus v)>(I\setminus p,M \setminus p)=\tau \setminus \sigma=\overline{\{\sigma,\tau\}}$ or $\overline{\{\sigma',\sigma\}} =\sigma\setminus\sigma'=(I \setminus v, M \setminus p \setminus v)>(I\setminus p,M \setminus p)=\tau \setminus \sigma=\overline{\{\sigma,\tau\}}$, since $p=\max I_{max}$. By Theorem \ref{thm:if_map_then_matched},  it follows that $\sigma$ is not critical.
\end{enumerate}

In all 4 cases, the weight of $\overline{\{\sigma,\tau\}}$ is less than the weight all other possible matching pair. By Theorem \ref{thm:if_map_then_matched},  it follows that $\sigma \rightarrow \tau$.

\end{proof}

\end{document}